\documentclass[11pt,oneside]{article}
\usepackage[utf8]{inputenc}
\usepackage[english]{babel}
\usepackage{amsmath}
\usepackage{amsfonts}
\usepackage{amssymb}
\usepackage{amsthm}
\usepackage[left=3cm,right=3cm,top=3cm,bottom=3cm]{geometry}

\theoremstyle{plain}
\newtheorem{teo}{}[section]
\newtheorem{prop}[teo]{Proposition}
\newtheorem{cor}[teo]{Corollary}
\newtheorem{rem}[teo]{Remark}
\newtheorem{lem}[teo]{Lemma}
\newtheorem{thm}[teo]{Theorem}
\newtheorem{df}[teo]{Definition}
\theoremstyle{definition}

\newcommand\blfootnote[1]{%
  \begingroup
  \renewcommand\thefootnote{}\footnote{#1}%
  \addtocounter{footnote}{-1}%
  \endgroup
}

\title{On the triviality of flows in Alexandroff spaces}
\author{Pedro J. Chocano, Diego Mond\'{e}jar Ruiz, Manuel A. Mor\'on and Francisco R. Ruiz del Portal}
\date{}

\begin{document}
\maketitle

\large

{\centering\footnotesize Dedicated to José Manuel Rodríguez Sanjurjo on the occasion of his 70th birthday.\par}

\begin{abstract}
We prove that the unique possible flow in an Alexandroff $T_{0}$-space is the trivial one. To motivate this result, we relate Alexandroff spaces to topological hyperspaces.
\end{abstract}
\blfootnote{This research is partially supported by Grants PGC2018-098321-B-100 and BES-2016-076669 from Ministerio de Ciencia, Innovación y Universidades (Spain).}
\section{Introduction and motivation}
In \cite{alexandroff1937diskrete}, P.S. Alexandroff introduced what he called  ``Diskrete Räume'' (``discrete spaces''), which are nowadays known as Alexandroff spaces. Today, the term discrete space refers to a topological space $(X,T)$ where the topology $T$ is the power set of $X$. An Alexandroff space $(X,T)$ is a topological space, where the topology $T$ satisfies the following stronger axiom: the arbitrary intersection of open sets is again an open set. P.S. Alexandroff also related the spaces introduced in \cite{alexandroff1937diskrete} to preordered sets. Let $(X,T)$ be an Alexandroff space and $x\in X$. The intersection of every open set containing $x$, denoted by $U_x$, is an open set. The open set $U_x$ is minimal in the following sense: $U_x\subseteq V$ for every open set $V$ containing $x$. It can be related to $X$ the preorder given by $x\leq y$ if and only if $U_x\subseteq U_y$ ($U_y\subseteq U_x$). The first preorder is said to be the natural preorder, while the second one is said to be the opposite preorder. We have the following well-known results, that can be found for example in \cite[Lemma 1.6.2 and Lemma 1.2.7]{may1966finite}.
\begin{lem}\label{thm:teoremaAlexandroffA} An Alexandroff space satisfies the Kolgomorov separation axiom (also known as $T_0$ axiom) if and only if the related preorder is a partial order.
\end{lem}
\begin{lem}\label{thm:teoremaAlexandroffB} An Alexandroff space satisfies the separation axiom $T_1$ if and only if it is a discrete space.
\end{lem}

Given a partially ordered set $(X,\leq)$, a lower (upper) set $S\subseteq X$ is a set satisfying that for every $x\in S$ and $y<x$ in $X$  ($y>x$) then $y\in S$. The family of lower (upper) sets of a partially ordered set $(X,\leq)$ is a $T_0$ topology on $X$ that makes $X$ an Alexandroff space. It is not difficult to prove that for every set $X$ the partial orders defined on $X$ are in bijective correspondence with the $T_0$ topologies defined on $X$ making $X$ an Alexandroff space. Furthermore, let $f:X\rightarrow Y$ be a map between two Alexandroff $T_0$-spaces. Then $f$ is continuous if and only if $f$ is order-preserving, where we are considering the natural partial orders in $X$ and $Y$. From these facts, the category of Alexandroff spaces and the category of partially ordered sets are isomorphic. From now on, we treat Alexandroff $T_0$-spaces and partially ordered sets as the same object without explicit mention. Moreover, given an Alexandroff $T_0$-space, $U_x$ can be described as follows: $U_x=\{y\in X|y\leq x \}$

The class of Alexandroff spaces contains the class of all finite topological spaces, i.e., topological spaces having finite cardinal. Some important properties of Alexandroff spaces were established by M.C. McCord in \cite{mccord1966singular}.
\begin{thm}\label{thm:mcCord1} For any Alexandroff space $X$, there exist a $CW$-complex $K$ and a weak homotopy equivalence $f_K:|K|\rightarrow X$. Moreover, if $X$ is path-connected, $K$ can be chosen path-connected.
\end{thm}
\begin{thm}\label{thm:mcCord2} For any $CW$-complex $K$, there exist an Alexandroff space $X$ and a weak homotopy equivalence $f_K:|K|\rightarrow X$. Moreover, if $K$ is path-connected, $X$ can be chosen path-connected.
\end{thm}

Flows (or continuous dynamical systems) are another central concept of this note. Recall that a flow on a topological space $X$ is a continuous map $\varphi:\mathbb{R}\times X\rightarrow X$ satisfying that $\varphi(0,x)=x$ for every $x\in X$ and $\varphi(s+t,x)=\varphi(s,\varphi(t,x))$ for every $t,s\in \mathbb{R}$ and $x\in X$, where $\mathbb{R}$ represents the real numbers with its usual topology. Note that if $\varphi_t:X\rightarrow X$ is the map $\varphi_t(x)=\varphi(t,x)$, then $\varphi_0$ is the identity map and $\varphi_s\circ \varphi_t=\varphi_{s+t}$. Consequently, $\varphi_t$ is a homeomorphism for every $t\in \mathbb{R}$, where the inverse of $\varphi_t$ is given by $\varphi_{-t}$.

Notice that if a topological space $(X,T)$ is discrete  and $\varphi:\mathbb{R}\times X\rightarrow X$ is a flow, then $\varphi$ is trivial, that is, $\varphi(t,x)=x$ for all $x\in X$, $t\in \mathbb{R}$. This is obvious because the paths in $X$ are constant.

Recently, it has been proved that topological spaces can be approximated by Alexandroff spaces using the concepts of inverse systems (inverse sequences) and inverse limits, see \cite{mondejar2020reconstruction} for compact metric spaces and \cite{bilski2017inverse} for locally compact, paracompact and Hausdorff spaces. One could ask if flows on a general space $(X,T)$ can be in some sense approximated by flows in sufficiently close elements in the corresponding approach of inverse systems mentioned above. Or, in other words, can Alexandroff spaces have sufficiently enough topological richness on paths to admit different flows on them in order to approximate different flows in the general topological space $(X,T)$? We were led to this question motivated by some ideas developed in the Ph. D. thesis of the first author \cite{chocano2021computational}  co-advised by the third  and fourth  authors of this note.
We recommend J.P. May's notes \cite{may1966finite} for general knowledge on Alexandorff spaces.

\section{Some previous related results and the answer to the question}
In view of Theorem \ref{thm:mcCord1} and Theorem \ref{thm:mcCord2}, the above question has some interest because there is a lot of non-discrete and non-finite Alexandroff spaces with very rich structure of continuous paths into them. Particularly, the fundamental group of such spaces can be chosen to be isomorphic to any prefixed group. On the other hand, every group can be realized as the homeomorphism group of an Alexandroff space, see \cite{chocano2020topological}.

Sometime ago, part of the authors of this note studied the so-called upper semicontinuous or lower semicontinuous topologies in hyperspaces of closed sets in a topological space $X$, see for instance \cite{cuchillo1992lower,cuchillo1999note,moron2008homotopical}. We recall briefly some of these definitions to make this note self-contained. Let $(X,T)$ be a topological space. Then $2^{X}=\{ E\subset X| E$ is closed and not empty$\}$. The upper (lower) semi-finite topology on $2^X$ is generated by taking as a basis (resp. sub-basis) for the open collections in $2^X$ all collections of the form $\{E\in 2^X|E\subset U \}$ (resp. $\{E\in 2^X|E\cap U\neq \emptyset \}$) with $U$ an open subset of $X$. For more details we refer the reader to \cite[Definition 9.1]{michael1951topologies}. Specifically, they proved in \cite[Proposition 1]{cuchillo1999note} that the unique possible flows in the lower hyperspace of closed sets of a topological space $(X,T)$, denoted by $2_L^X$, are those constructed in a natural way by means of genuine flows on $(X,T)$. As an immediate consequence, we get the following result.
\begin{cor}\label{cor:corolarioflujoenespaciotrivial} Let $X_d$ be a topological space with the discrete topology. Then the unique possible flow on $2_L^{X_d}$ is the trivial one.
\end{cor}
The reason why Corollary \ref{cor:corolarioflujoenespaciotrivial} holds is because any flow $\varphi:\mathbb{R}\times 2_L^{X_d}\rightarrow 2_L^{X_d}$ comes naturally from a flow $\pi:\mathbb{R}\times X_d\rightarrow X_d$. Therefore, for any $t\in \mathbb{R}$ and $x\in X$ we have a continuous path $\pi_t^x:[0,t]\rightarrow X$ given by $\pi_t^x(s)=\pi(s,x)$, $\pi_t^x(0)=x$ and $\pi_t^x(t)=\pi(t,x)$.  Since $X_d$ is a discrete topological space, it follows that the unique possible paths are just $\pi_t^x(s)=x$ for all $t\in \mathbb{R}$ and $s\in [0,t]$. Thus, the flow $\pi$ is trivial, and consequently $\varphi$ is trivial.

To relate the above result with our current framework, we have the following proposition.
\begin{prop}\label{prop:AlexandroffFinito} Let $X_d$ be a discrete topological space. Then $2_L^{X_d}$ is an Alexandroff space if and only if $X$ is a finite set.
\end{prop}
\begin{proof}
Suppose that $X$ is a finite set and $C\subset X$ is a non-empty subset. Then $C=\{c_1,...,c_k \}$, where $k\in \mathbb{N}$. Consequently, $U_C=L_{\{c_1 \}}\cap ... \cap L_{\{c_k\}}=\{ D\subset X|C\subset D\}$ is open in $2^{X_d}_L$ and obviously is the minimal open neighborhood of $C$ in $2^{X_d}_L$. In the above notation, $L_U$ for an open set $U$ in $X$ means $L_U=\{ F\subset X|F\neq \emptyset$ and $F\cap U\neq \emptyset \}$. Now, suppose $2^{X_d}_L$ is an Alexandroff space and consider $X\in 2^{X_d}_L$. Note firstly that for any non-empty open set $\mathcal{U}\subset 2^{X_d}_L$, we have $X\in \mathcal{U}$. This is because if $C\in \mathcal{U}$, then there are non-empty open subsets $U_1,...,U_m$ of $X$ such that $C\in L_{U_1}\cap... \cap L_{U_m}\subset \mathcal{U}$. Obviously, by definition, $X\in L_{U_1}\cap... \cap L_{U_m}$. Then $X$ is an element in the intersection of all non-empty open sets in $2^{X_d}_L$. In fact, we can prove that $\{X\}=\bigcap_{\mathcal{U}\neq \emptyset \text{ open in }2^{X_d}_L}\mathcal{U}$. This is because if $F\subset X$, where $F\neq X$, then there is an element $f\in X$ such that $f\notin F$, which means that $F\notin L_{\{f\}}$. Note that $L_{\{f\}}$ is open in $2^{X_d}_L$ because $\{ f\}$ is open in $X_d$. We have proved that $\{X\}=\bigcap_{\mathcal{U}\neq \emptyset \text{ open in }2^{X_d}_L}\mathcal{U}$. Hence, the unitary subset $\{ X\}$ is open in $2^{X_d}_L$ if  $2^{X_d}_L$ were an Alexandroff space. Moreover, in this case, $\{ X\}$ is the minimal open neighborhood of the point $X$ in $2^{X_d}_L$. Following \cite[Definition 9.1 ]{michael1951topologies}, we have that there exist non-empty subsets $U_1,...,U_l$ of $X$ with $\{ X\}=L_{U_1}\cap... \cap L_{U_l}$, where we can assume without loss of generality that $U_i\neq U_j$ if $j\neq i$. This condition implies that there is a finite subset $F\subset \bigcup_{i=1}^l U_i$ with $F\cap U_j\neq \emptyset$ for all $j=1,...,l$. Consequently, $F\in L_{U_1}\cap ... \cap L_{U_l}$ and then $F=X$. Thus, $X$ is finite.
\end{proof}
Proposition \ref{prop:AlexandroffFinito} motivates us to give the following definition.
\begin{df} Suppose $X_d$ is a discrete topological space, $U\subset X$ and $L_U=\{F\subset X|F\neq \emptyset$ and $F\cap U\neq \emptyset \}$. The family $SL=\{\bigcap_{i\in I} L_{U_i}|I$ is any set with $Card(I)\leq Card(X)\}$ is a base for a topology in $2^X$. We represent by $2^X_{SL}$ the corresponding topological space and we call it the hyperspace of $X$ with the strong lower semifinite topology.
\end{df}

\begin{prop}\label{prop:propiedades} Suppose $X$ is any set and consider it as a topological space with the discrete topology, denoted by $X_d$. Then the following properties are satisfied.

\begin{enumerate}
\item $2^{X_d}_{SL}$ is always an Alexandroff $T_{0}$-space.
\item For any $C\in 2^X$, the minimal neighborhood of $C$ in $2^{X_d}_{SL}$ is given by $\bigcap_{c\in C} L_{\{ c\}}=\{ D\in 2^X|C\subset D\}$.
\item The partial ordered set associated to the Alexandroff $T_{0}$-space $2^{X_d}_{SL}$ is given by $C\geq D$ if and only if $C\subseteq D$.
\item The identity map $I:2^{X_d}_{SL}\rightarrow 2^{X_d}_{L}$ is continuous and the inverse is continuous at $C\in 2^X$ if and only if $C$ is a finite set.
\item If we consider in $2^X$ the opposite order to that considered in 3., the corresponding Alexandroff $T_{0}$-space we get is $2^{X_d}_U$, that is,  the hyperspace of the discrete space $X_d$ with the upper semifinite topology.
\end{enumerate}

\end{prop}

With all above, we get the following four families of Alexandroff $T_{0}$-spaces:
\begin{itemize}
\item $\mathcal{F}_1=\{ 2^{X_d}_{SL}$ for a discrete topological space $X_d\}$.
\item $\mathcal{F}_2=\{ 2^{X_d}_{U}$ for a discrete topological space $X_d\}$.
\item $\mathcal{F}_3=\{ 2^{X_d}_{F_L}$ for a discrete topological space $X_d\}$.
\item $\mathcal{F}_4=\{ 2^{X_d}_{F_U}$ for a discrete topological space $X_d\}$.
\end{itemize}
Where $2^X_F=\{ C\subset X|C\neq \emptyset$ and $Card(C)<\infty\}$ and $2^{X_d}_{F_L}$ is the topological subspace restricting the lower semifinite topology to $2^{X_{d}}_F$. Analogously, we have $2^{X_d}_{F_U}$ for the upper semifinite topology.

\begin{rem}
\begin{itemize}
\item The families  $\mathcal{F}_1$ and $\mathcal{F}_2$ can be used to embed, topologically, any Alexandroff $T_{0}$-space.
\item The families  $\mathcal{F}_3$ and  $\mathcal{F}_4$ can be used to embed, topologically, any locally finite Alexandroff $T_{0}$-space.
\end{itemize}
Results related to above statements can be found in \cite[Chapter 4]{mondejar2015hyperspaces}.
\end{rem}

One could follow proving \cite[Proposition 1]{cuchillo1999note} for spaces in any of the families $\mathcal{F}_1,\mathcal{F}_2,\mathcal{F}_3,\mathcal{F}_4$ obtaining that any space in $\mathcal{F}_1\cup \mathcal{F}_2\cup \mathcal{F}_3\cup\mathcal{F}_4$ is an example of an Alexandroff space admitting only trivial flows. Instead of this, we finish this note giving the following general result.
\begin{thm} Let $X$ be any Alexandroff $T_0$-space and  let $\varphi:\mathbb{R}\times X\rightarrow X$ be a flow. Then $\varphi$ is trivial, i.e., $\varphi(t,x)=x$ for any $t\in \mathbb{R}$ and $x\in X$.
\end{thm}
\begin{proof}
For any $x\in X$ consider $U_x$ as the minimal open neighborhood of $x$ in $X$. In terms of the corresponding partial order $\leq $ in $X$, we have $U_x=\{ y\in X|y\leq x\}$. Using the product topology in $\mathbb{R}\times X$, the continuity of $\varphi$ and the fact that $\varphi(0,x)=x$ for all $x\in X$, we deduce that for any $x\in X$ there is a positive value $t_x$ such that $\varphi((-t_x,t_x)\times U_x)\subset U_x$. Consider now $t$ such that $|t|<t_x$. Then the map $\varphi_{t_{|U_x}}$ given by $\varphi_{t_{|U_x}}(y)=\varphi(t,y)$ for $y \in U_x$ is surjective. In fact, it is a homeomorphism from $U_x$ onto itself. We prove the last assertion. Suppose $y\in U_x$. Then $y=\varphi_{0_{|U_x}}(y)=\varphi_{t_{|U_x}}\circ \varphi_{-t_{|U_x}} (y)$. Moreover, $\varphi_{-t_{|U_x}}(y)\in U_x$ because $|t|<t_x$, which implies that $\varphi_{t_{|U_x}}(\varphi_{-t_{|U_x}}(y))=y$.

Now, consider the point $(t,x)\in (-t_x,t_x)\times U_x$. We will prove that $\varphi(t,x)=x$. We have that $x=\varphi_t(\varphi_{-t}(x))$ and $\varphi_{-t}(x)\in U_x$. If $\varphi_{-t}(x)\neq x$, then $\varphi_{-t}(x)<x$ and $\varphi_{t}(x)\neq x$ due to the injectivity of $\varphi_{t}$. By the continuity of $\varphi_{t}$, $\varphi_{t}$ preserves the order, so $x=\varphi_{t}(\varphi_{-t}(x))<\varphi_t(x)$ and then $x\in U_{\varphi_{t}(x)}$. On the other hand, $\varphi_{t}(x)\in U_{x}$ because $|t|<t_x$. This implies that $U_{\varphi_{t}(x)}= U_{x}$ and it is a contradiction because $X$ is a $T_{0}$-space.
Hence, we have $\varphi_{-t}(x)=x$. Since $\varphi_t^{-1}=\varphi_{-t}$, it follows $\varphi_t(x)=\varphi(t,x)=x$ for any $t$ satisfying $|t|<t_x$. Finally, fix now $(s,x)\in \mathbb{R}\times X$. Then there is a natural number $n$ such that $|\frac{s}{n}|<t_x$. Therefore, $\varphi(\frac{s}{n},x)=x$, but $s=\sum_{k=1}^{n}\frac{s}{n}$. Since $\varphi$ is a flow, and using  the associativity of the sum of real numbers, we have $\varphi(s,x)=\varphi(\sum_{k=1}^{n}\frac{s}{n},x)=\varphi(\sum_{k=1}^{n-1}\frac{s}{n},\varphi(\frac{s}{n},x))=\cdots=\varphi(\frac{s}{n},x)=x$ and then $\varphi$ is trivial .
\end{proof}
\begin{rem} Notice that this result may be seen as a generalization of \cite[Lemma 8.1.1]{barmak2011algebraic}. This lemma states the following: if $X$ is a finite $T_0$-space, $x\in X$ and $f:X\rightarrow X$ is a homeomorphism satisfying that $x$ is comparable to $x$, then $f(x)=x$.  
\end{rem}

\textbf{Acknowledgments.} The authors wish to express their thanks to the reviewer and editor for their helpful comments.

\bibliography{bibliography}
\bibliographystyle{plain}

\newcommand{\Addresses}{{
  \bigskip
  \footnotesize
  
  \textsc{ P.J. Chocano, Departamento de Matemática Aplicada,
Ciencia e Ingeniería de los Materiales y
Tecnología Electrónica, ESCET
Universidad Rey Juan Carlos, 28933
Móstoles (Madrid), Spain}\par\nopagebreak
  \textit{E-mail address}:\texttt{pedro.chocano@urjc.es}

  \medskip

    \textsc{ D. Mond\'{e}jar Ruiz,  Departamento de Matemática Aplicada y Estadística, Facultad de Ciencias Económicas y
Empresariales, Universidad San Pablo-CEU, CEU Universities, Calle Julián Romea 23, 28003
Madrid, Spain}\par\nopagebreak
  \textit{E-mail address}: \texttt{diego.mondejarruiz@ceu.es}

  \medskip

\textsc{ M. A. Mor\'on,  Departamento de \'Algebra, Geometr\'ia y Topolog\'ia, Universidad Complutense de Madrid and Instituto de
Matematica Interdisciplinar, Plaza de Ciencias 3, 28040 Madrid, Spain}\par\nopagebreak
  \textit{E-mail address}: \texttt{ma\_moron@mat.ucm.es}

  \medskip

\textsc{ F. R. Ruiz del Portal,  Departamento de \'Algebra, Geometr\'ia y Topolog\'ia, Universidad Complutense de Madrid and Instituto de
Matematica Interdisciplinar
, Plaza de Ciencias 3, 28040 Madrid, Spain}\par\nopagebreak
  \textit{E-mail address}: \texttt{R\_Portal@mat.ucm.es}

}}

\Addresses

\end{document}